\documentclass[11pt,a4paper]{amsart}

\usepackage[latin1]{inputenc}
\usepackage{amsmath,amsfonts,amssymb,setspace,amscd}
\usepackage{enumerate}
\usepackage[
      colorlinks=true,    
      urlcolor=blue,    
      menucolor=blue,    
      linkcolor=blue,    
      bookmarks=true,    
      bookmarksopen=true,    
      citecolor=blue,
      hyperfootnotes=false,    
      pdfpagemode=UseOutlines    
]{hyperref}
\usepackage[top=3 cm, bottom=3 cm, left=3 cm, right=3 cm]{geometry}
\onehalfspace

\usepackage{pdfsync,mathdots}
\usepackage{graphicx}

\newtheorem{theorem}{Theorem}[section]

\newtheorem{proposition}[theorem]{Proposition}

\theoremstyle{definition}

\newtheorem{remark}[theorem]{Remark}
\newtheorem{example}[theorem]{Example}

\numberwithin{equation}{section}
\numberwithin{theorem}{section}

\newcommand{\N}{\mathbb{N}}

\newcommand{\Z}{\mathbb{Z}}
\newcommand{\Q}{\mathbb{Q}}

\newcommand{\F}{\mathbb{F}}

\newcommand{\Qbar}{\overline{\Q}}

\newcommand{\ord}{{\rm{ord}}}

\title[A note on cyclotomic polynomials and LFSR]{A note on cyclotomic polynomials and \\ Linear Feedback Shift Registers}
\author{Laura Capuano}
\address{DISMA ``Luigi Lagrange'', Politecnico di Torino, Corso Duca degli Abruzzi 24, 10129 Torino, Italy}
\email{laura.capuano@polito.it}

\author{Antonio J. Di Scala}
\address{DISMA ``Luigi Lagrange'', Politecnico di Torino, Corso Duca degli Abruzzi 24, 10129 Torino, Italy}
\email{antonio.discala@polito.it}

\begin{document}

\maketitle
\begin{abstract}
Linear Feedback Shift Registers (LFRS) are tools commonly used in cryptography in many different context, for example as pseudo-random numbers generators. In this paper we characterize LFRS with certain symmetry properties. Related to this question we also classify polynomials $f$ of degree $n$ satisfying the property that if $\alpha$ is a root of $f$ then $f(\alpha^n)=0$. The classification heavily depends on the choice of the fields of coefficients of the polynomial; we consider the cases $K=\F_p$ and $K=\Q$.
\end{abstract} 

\section{Introduction}

The motivation of this paper comes from an exercise in a written exam of Cryptography about the bit stream $(s_j)$ ($j=0,1, \cdots$) generated by a (Fibonacci $n$-bit) Linear Feedback Shift Register (LFSR) see e.g. \cite[page 374]{Schneier15}:

\begin{figure}[h!]
\begin{center}
\includegraphics[height=3cm,width=10cm]{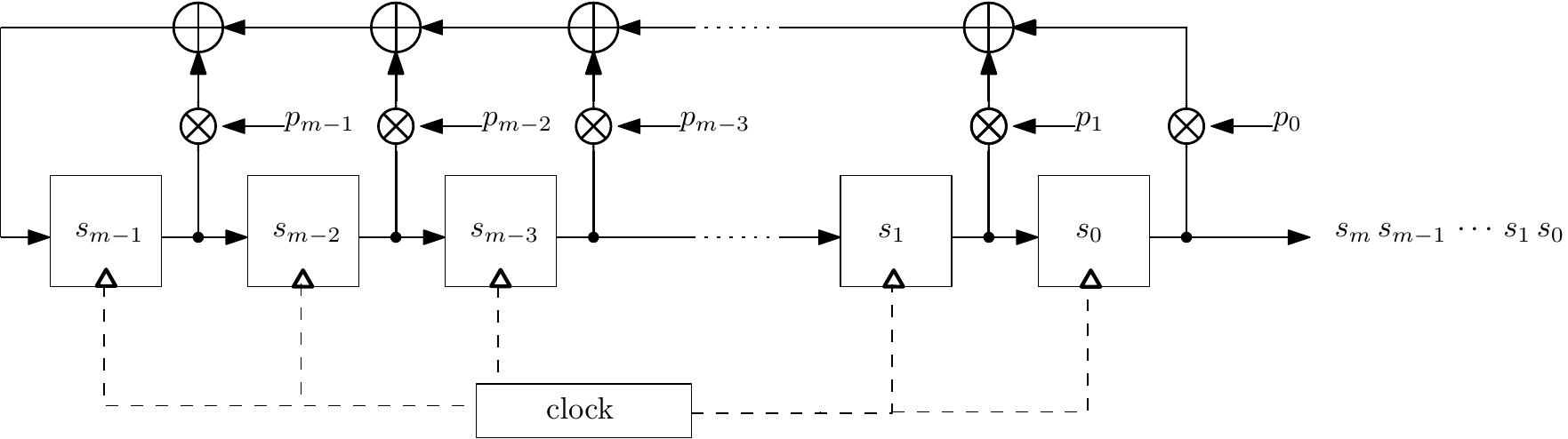}
\end{center}
\caption{a Fibonacci LFSR}
\end{figure}

A linear feedback shift register (LFSR) is a shift register whose input bit is a linear function of its previous state.
The bit stream  
\[ 
\cdots s_n \, s_{n-1} \, \cdots \, s_2 \,  s_1 \,  s_0
\]
is generated recursively as in the Fibonacci sequence; namely, given the initialization state of the register $s_{n-1}, \cdots, s_0 \in \F_2$, for every $j \geq n$
 the bit $s_j$ is computed recursively as
\[ 
s_j = s_{j-1} p_{n-1} + s_{j-2} p_{n-2} + \cdots + s_{j-n} p_0  \, \, (\mathrm{mod} \, \, 2). 
\]
The polynomial $\chi(x) = x^n + p_{n-1}x^{n-1} + \cdots + p_1 x + p_0$ characterizes the LFSR.

Because the operation of the register is deterministic, the stream of values produced by the register is completely determined by its current (or previous) state. Likewise, because the register has a finite number of possible states, it must eventually enter a repeating cycle. However, a LFSR with a well-chosen feedback function can produce a sequence of bits which has a very long cycle and so appears random. Applications of LFSRs include generating pseudo-random numbers, pseudo-noise sequences, fast digital counters, and whitening sequences.

\medskip

By using a row $\mathrm{w}=[s_{n-1} s_{n-2} \cdots s_1 s_0 ]$ to describe the state of the register, we have that 
it changes according to a right multiplication $ \mathrm{w} \cdot \mathrm{L}$, where $\mathrm{L}$ is the $n \times n$ matrix
\begin{equation}\label{LFSRm} 
\mathrm{L} =
\begin{bmatrix}
{ p_{n-1}}  & {1} & 0 & 0 & \cdots & 0 \\
{ p_{n-2}}  & 0 & {1} & 0 & \cdots & 0 \\
{ p_{n-3}}  & 0 & 0 & {1} & \cdots & 0 \\
\vdots  &\vdots & \vdots & \vdots & \ddots & \vdots\\
{ p_1 } & 0 & 0 & 0 & \cdots & {1} \\
{ p_0 } & 0 & 0 & 0 & \cdots & 0 \\
\end{bmatrix}
\end{equation}

Notice that $\chi(x)$ is exactly the characteristic polynomial of $\mathrm{L}$.

\medskip

The above mentioned exercise asked to run the $2$-bit LFSR with polynomial $\chi(x)=x^2 + x + 1$ and to compute the first $6$ bits
$s_5,s_4,s_3,s_2,s_1,s_0$ given $s_1 = 0 , s_0 = 1$.
What grabbed our attention and motivated this paper is that a student constructed the correct bit stream but in a wrong way. Namely, he constructed a bit stream $(r_j)$ by computing \[ [r_j \, r_{j+1} ] = [r_{j-2} \,  r_{j-1}] \cdot  \mathrm{L} \, .\]
It turns out that the bit streams $(r_j)$ and $(s_j)$ are the same because, in this specific case, the matrix associated to the LFSR satisfies
\begin{equation}{\label{2-bit}}
 \mathrm{L} = \tau \cdot \mathrm{L}^2 \cdot \tau,
\end{equation}
where $\tau$ is the $2 \times 2$ reflection matrix $\tau = \begin{bmatrix} 0 & 1 \\ 1 & 0 \end{bmatrix}$. 

 It is then natural to ask, more in general, about the classification of $n$-bit LFSR such that
 \begin{equation}{\label{n-bit}} \mathrm{L} = \tau \cdot \mathrm{L}^n \cdot \tau, \end{equation}
where $\tau = \begin{bmatrix} 0 & \cdots & 1 \\ 0 & \iddots & 0 \\ 1 & \cdots & 0 \end{bmatrix}$ is the $n \times n$ reflection matrix. We will prove the following result:

\begin{theorem}\label{thm:LFSRthm} Let $\mathrm{L}$ be a $n \times n$ matrix with entries in $\F_2$ of the form (\ref{LFSRm}) and let $\chi(x)$ be its characteristic polynomial. Then, $\mathrm{L}$ satisfies equation (\ref{n-bit}) if and only if $\chi(x)= \sum_{j=0}^{n} x^j$.
\end{theorem}

More in general, \eqref{n-bit} implies that the matrices $\mathrm{L}$ and $\mathrm{L}^n$ are similar, and so they have the same eigenvalues. In particular, this implies that for every root $\alpha \in \overline{\F}_2$ of $\chi(x)$, then $\chi(\alpha^n)=0$. This led to the following natural question: given a field $K$, classify all the polynomials $f(x)\in K[x]$ of degree $n$ satisfying the property
\begin{equation} \label{main_property}
\mbox{for every } \alpha \in \overline{K} \mbox{ such that } f(\alpha)=0, \mbox{ then } f(\alpha^n)=0.
\end{equation}
Notice that the polynomial $\chi(x)=\sum_{j=0}^n x^j$, arising from Theorem \ref{thm:LFSRthm}, satisfies this property.

Trivially, every linear polynomial satisfies \eqref{main_property}, so we will always assume without loss of generality that $\deg f \ge 2$.
If $f$ satisfies the property, then for every $i\ge 0$, $\alpha^{n^i}$ is a root of $f$. As $f$ has at most $n$ distinct roots, this implies that there exist $0\le h < k \le n$ such that $\alpha^{n^k}= \alpha^{n^h}$. If we rewrite this equality as $\alpha^{n^h}(\alpha^{{n^h}(n^{k-h}-1)}-1)=0$, then we have that either $\alpha=0$ or $\alpha$ is a root of unity. This property is independent of the nature of the field of coefficients $K$, while the characterization heavily depends on it.

In this paper we will give the full classification in the case  $K=\Q$ (Section \ref{sec:rational}) and for irreducible polynomials in the case $K=\F_p$ (Section \ref{sec:finite_fields}). In both cases, these polynomials are strictly connected to the cyclotomic polynomials, and in particular it turns out their degrees are connected with other classical problems in number theory, e.g. the classification of the numbers which are coprime with their Euler totient function, which arise in several different contexts, as explained in Section \ref{sec:Michon_numbers}.

\subsection*{Acknowledgements}
Both the authors are members of the INdAM group GNSAGA, of CrypTO (the group of Cryptography and Number Theory of Politecnico di Torino), and of DISMA, Dipartimento di Eccellenza MIUR 2018-2022.

\section{Proof of Theorem \ref{thm:LFSRthm}}

\noindent In this section we prove Theorem \ref{thm:LFSRthm}. 

\medskip 

\noindent First, assume that $\mathrm{L} = \tau \cdot \mathrm{L}^n \cdot \tau$; then  $\tau \cdot \mathrm{L} = \mathrm{L}^n \cdot \tau$. 
A direct computation shows that: 
\[ 
\tau \cdot \mathrm{L} =
\begin{bmatrix}
{ p_{0}}     & 0 & * & * & \cdots & * \\
{ p_{1}}    & 0 & * & * & \cdots & * \\
{ p_{2}}    & 0 & * & * & \cdots & * \\
\vdots      &\vdots & \vdots & \vdots & \ddots & \vdots\\
{ p_{n-2} } & 0 & * & * & \cdots & * \\
{ p_{n-1} } & 1 & * & * & \cdots & * \\
\end{bmatrix}, 
\]
and
\[ 
\mathrm{L}^n \cdot \tau =
\begin{bmatrix}
{ p_{n-1}} & p_{n-2} + p_{n-1}^2 & * & * & \cdots & * \\
{ p_{n-2}} & p_{n-3} + p_{n-2} p_{n-1} & * & * & \cdots & * \\
{ p_{n-3}} & p_{n-4} + p_{n-3} p_{n-1}  & * & * & \cdots & * \\
\vdots     &\vdots & \vdots & \vdots & \ddots & \vdots\\
{ p_{1} }  & p_0 + p_1 p_{n-1} & * & * & \cdots & * \\
{ p_{0} }  & p_0 p_{n-1} & * & * & \cdots & * \\
\end{bmatrix},
\]
which holds in any characteristic.
From these formulas, as the coefficients of the characteristic polynomial lie in $\F_2$, we directly get that $p_0=p_1=\cdots=p_{n-1}=1$, i.e. $\chi(x)=\sum_{j=0}^n x^j$.
\medskip

\noindent Conversely, assume that $\chi(x)=\sum_{j=0}^n x^j$ i.e. \[ \mathrm{L} =
\begin{bmatrix}
{ 1}  & {1} & 0 & 0 & \cdots & 0 \\
{ 1}  & 0 & {1} & 0 & \cdots & 0 \\
{ 1}  & 0 & 0 & {1} & \cdots & 0 \\
\vdots  &\vdots & \vdots & \vdots & \ddots & \vdots\\
{ 1 } & 0 & 0 & 0 & \cdots & {1} \\
{ 1} & 0 & 0 & 0 & \cdots & 0 \\
\end{bmatrix} \, .\]
Let us write $\mathrm{L} = [E_0 E_1 E_2 \cdots E_{n-1}]$, where $E_j$ denote the $(j+1)^{\rm{th}}$ column of $\mathrm{L}$. Notice that, for every matrix $\mathrm{M} = [M_1 M_2 \cdots M_n]$, then the following holds: 
\[ \mathrm{M} \cdot \mathrm{L} = \left [ \left ( \sum_j M_j \right ) \, M_1 \, M_2 \, \cdots \, M_{n-1} \right ] \, . \]
Using this property with $M=L$, it follows that
\[\mathrm{L}^2 = [E_n \, E_0 \, E_1 \, \cdots \, E_{n-2}] \]
where $E_n = (\sum_{j=0}^{n-1} E_j)$.
Since we are working in $\F_2$, a simple inductive argument implies that, for every $3 \le k \le n$,
\[
\mathrm{L}^k = [E_{n-k+2} \, E_{n-k+3} \, \cdots \, E_n \, E_0 \, \cdots \, E_{n-k-1}];
\]
in particular
\[ 
 \mathrm{L}^n = [E_2 E_3 \cdots E_{n-1} E_n E_0] \, . 
\]
Now it is straightforward to check that $ \tau \cdot \mathrm{L} = \mathrm{L}^n \cdot \tau$, as wanted. \qed

\begin{remark}
We point out that the statement of Theorem \ref{thm:LFSRthm} heavily depends on the fact that the field of the coefficients has characteristic $2$. More in general, it can be proved in the same way that, if char$(K)\neq 2$, a matrix $L$ of the form \eqref{LFSRm} satisfies $\mathrm{L}=\tau \cdot \mathrm{L}^n \cdot \tau$ if and only if 
\begin{equation*}
\mathrm{L} =
\begin{bmatrix}
{-1}  & {1} & 0 & 0 & \cdots & 0 \\
{-1}  & 0 & {1} & 0 & \cdots & 0 \\
{-1}  & 0 & 0 & {1} & \cdots & 0 \\
\vdots  &\vdots & \vdots & \vdots & \ddots & \vdots\\
{-1} & 0 & 0 & 0 & \cdots & {1} \\
{-1} & 0 & 0 & 0 & \cdots & 0 \\
\end{bmatrix}
\end{equation*}
Notice that, if  char$(K)\neq 2$, then a matrix $\mathrm{L}$ of the form \eqref{LFSRm} has characteristic polynomial equal to $\chi(x)=x^n-p_{n-1}x^{n-1}-\cdots -p_0$. This implies that the previous matrix has characteristic polynomial $\chi(x)=\sum_{i=0}^n x^i$, which is a product of cyclotomic polynomials. 
\end{remark}

\section{Cyclotomic polynomials} \label{sec:cyclotomic}

In this section we recall the definition and some basic properties of cyclotomic polynomials which will be used later. For more references, see \cite{Lang} or \cite{finite_fields}.

Given an integer $n\ge 0$, we let $\Phi_n(x)$ denote the $n^{th}$-cyclotomic polynomial, defined by
\[
\Phi_n(x):=\prod_{\substack{ 1\le k \le n \\ (k,n)=1} }\left (x- e^{\frac{2\pi i k}{n}} \right ).
\]
Clearly, we have that $x^n -1 = \prod_{d \mid n} \Phi_d (x)$ for every $n \ge 1$, and the M\"obius inversion formula gives 
\[
\Phi_n(x)=\prod_{d \mid n} \left ( x^d-1 \right )^{\mu(n/d)}= \prod_{d \mid n} \left ( x^{n/d}-1 \right )^{\mu(d)},
\]
where $\mu(d)$ denotes the M\"obius function, i.e. 
\[
\mu(d):= \begin{cases} 1 \ \ \quad \mbox{if } d \mbox{ is the product of a even number of distinct prime factors;}		\\
0 \ \ \quad \mbox{if } d \mbox{ is not squarefree;}\\
-1 \quad \mbox{if } d \mbox{ is the product of a odd number of distinct prime factors.} 
\end{cases}
\] 
It can be proved that, for every $n \ge 0$, $\Phi_n$ is a monic polynomial with integer coefficients and that $(\Phi_m(x), \Phi_n(x))=1$ for every $m<n$. The degree of $\Phi_n$ is clearly equal to $\varphi(n)$, which denotes the Euler totient function, i.e.  
\[
\varphi(n)= n \prod_{p \mid n} \left ( 1- \frac{1}{p} \right ).
\]

\begin{example}
If $n=p$ is a prime, then $\Phi_p(x)=x^{p-1}+ x^{p-2} + \cdots + 1$.
\end{example}

Let $p$ be a prime and let $\F_q$ denote the finite field with $q=p^f$ elements for some $f\ge 1$. The cyclotomic polynomials are irreducible over the field of rational numbers, but this is not the case over finite fields; indeed, the polynomial $\Phi_n$ is irreducible over $\F_q$ if and only if $n=2, 4, r^k$ or $2r^k$, where $r$ is an odd prime, $k$ is a positive integer and $q$ is a generator of $(\Z/n \Z)^{\times}$.\footnote{In particular this implies that if $(\Z/n\Z)^{\times}$ is not cyclic (i.e. unless $n$ is an odd prime power, twice an odd prime power, or $n=2$ or $4$), then $\Phi_n$ is a polynomial which is reducible modulo every prime $p$ but is irreducible over $\Q$.}

If $(n,q)=1$, then $\Phi_n$ can be factorized into $\varphi(n)/m$ distinct irreducible polynomials of the same degree $m$ over $\F_q$, where $m$ is the multiplicative order of $q$ modulo $n$, i.e. the least positive integer such that $q^m \equiv 1 \mod n$. For the rest of the paper we will denote this order by $\ord_n(q)$.
 If $(n,q)\neq 1$ we can write $n=p^a n'$ with $(p,n')=1$; then, we have that $\Phi_{n}(x)=(\Phi_{n'}(x))^{p^a}$, so it is enough to study the factorization for cyclotomic polynomials of roots of unity of order coprime with the characteristic of the field.

\begin{example} \label{ese:1}
Let us consider the cyclotomic polynomial of order $15$ over $\F_2$; we have that $\varphi(15)=8$ and $\ord_{15}(2)=4$, hence $\Phi_{15}$ factorises into $2$ factors of degree $4$, i.e.
\[
\Phi_{15}(x)=(x^4+x^3+1)(x^4+x+1).
\]
If we take the cyclotomic polynomial of order $8$ over $\F_2$, then we have that
\[
\Phi_8(x)=(x+1)^4.
\]

\end{example}

\section{Classifying the polynomials with rational coefficients satisfying property \eqref{main_property}}

Let $f(x) \in K[x]$ be a nonzero polynomial of degree $n$ with coefficients in a field $K$; we are interested in characterizing the polynomials with the following property:
\begin{equation*} 
\mbox{for every } \alpha \in \overline K \mbox{ such that } f(\alpha)=0, \mbox{ then } f(\alpha^n)=0.
\end{equation*}


We will be interested in two main cases, namely the case $K=\F_p$ and the case $K=\Q$. For every choice of the field $K$, we will first look at the irreducible case and then at the general case. Notice that, eventually dividing by the leading coefficient of $f$, it is enough to consider $f$ monic.

\subsection{The case of irreducible polynomials over $\Q$}

We saw at the end of the introduction that trivially all the polynomials of degree $1$ satisfy the condition, so without loss of generality we can assume $\deg f \ge 2$. In this case, if $f$ is a polynomial of degree $\ge 2$ having this property, then its roots are either $0$ or roots of unity. If $f$ is irreducible and monic, this means that $f(x)$ is the minimal polynomial of a (primitive) root of unity, i.e. a cyclotomic polynomial.

%
%
%

We can prove the following proposition, which gives the characterization of the irreducible polynomials which satisfy \eqref{main_property}.

\begin{proposition} \label{prop:irreducible}
Let $f$ be a monic, irreducible polynomial of degree $n\ge 2$; then, $f$ satisfies \eqref{main_property} if and only if $f(x)=\Phi_d(x)$ for some $d\ge 3$ with $(n,d)=1$.
\end{proposition}

\begin{proof}
If $f$ satisfies the property \eqref{main_property}, then for every root $\alpha$ of $f$, either $\alpha=0$ or $\alpha$ is a root of unity.

As we are assuming $f$ irreducible of degree $\ge 2$, then $f(x)$ is the minimal polynomial of a root of unity, i.e. a cyclotomic polynomial. Moreover, if the order of $\alpha$ is $d$, then $f(x)=\Phi_d(x)$ and $\deg f=\varphi(d)$, so we have to take $d\ge 3$ because we are assuming that $d\ge 2$. As we know that all the roots of $\Phi_d$ are primitive $d^{th}$-roots of unity, then $\alpha^n$ has to be a primitive $d^{th}$ root of unity, which implies that $(n,d)=1$, proving the first implication.
\medskip

We want now to prove the converse. Let us assume then that $f(x)=\Phi_d(x)$ is the $d^{th}$-cyclotomic polynomial with $d\ge 3$ and that the degree of $f$ is coprime with $d$. First, as $d \ge 3$, then $\deg f=\phi(d)$ as wanted. Moreover, if $\alpha$ is a root of $f$, then it is a primitive $d^{th}$-root of unity, and as $(n,d)=1$, we have that $\alpha^n$ is also a primitive $d^{th}$-root of unity, and so $f(\alpha^n)=0$, proving that $f$ satisfies the condition \eqref{main_property}, as wanted.
\end{proof}

\begin{example}\label{ese_sec4} \
\begin{itemize}
  \item If $n=p-1$ where $p$ is a prime, then $\Phi_p(x)=x^{p-1}+ \cdots + 1$ satisfies the condition \eqref{main_property};
  \item Let $\deg f=60$; there are exactly two polynomials of degree $60$ which are cyclotomic polynomials and satisfy the condition \eqref{main_property}, i.e. $\Phi_{61}(x)$ and $\Phi_{77}(x)$ (and $60$ is the smallest degree with this property).
 \end{itemize}
\end{example}

We just proved that the irreducible polynomials which satify the condition \eqref{main_property} are exactly the polynomials of degree $1$ and the cyclotomic polynomials $\Phi_d(x)$ with $(\deg f, d)=1$. In particular, since $\deg \Phi_d(x)=\varphi(d)$, where $\varphi$ is the Euler totient function, we are asking that $(d, \varphi(d))=1$. This also implies that if $\deg f >1$, then it has to be odd and squarefree (see Proposition \ref{prop:cyclic_char}).

The integers $d$ such that $(d, \varphi(d))=1$ have been deeply studied in the literature and appear in many different contexts. We are going to describe some properties of these numbers Section \ref{sec:Michon_numbers}.

\subsection{The general case over $\Q$} \label{sec:rational}

We want now to analyse the general case and characterise the polynomials $f$ (not necessarily irreducible) such that $f$ satisfies the property \eqref{main_property}.

As seen before, if $\alpha$ is a root of $f$ then either $\alpha$ is zero or $\alpha$ is a root of unity, hence either $x \mid f$ or, if $\alpha$ is a primitive $k^{th}$ root of unity, $\Phi_k \mid f$.

We give the following general characterization for the polynomials which satisfy property \eqref{main_property}:

\begin{theorem} \label{thm:general}
Let $f(x)\in \Q[x]$ be a monic polynomial of degree $n\ge 2$. Then, $f$ satisfies the property \eqref{main_property} if and only if
\begin{equation} \label{eq:shape}
f(x)=x^a \prod_{\substack{ 1\le h_1 < \cdots < h_r \le n \\ (n,h_i) = 1}} \Phi_{h_i}(x)^{b_i} \prod_{\substack{ 1\le k_1 < \cdots < k_s \le n \\ (n,k_j) \neq 1}} 
 \Phi_{k_j}(x)^{c_j} \prod_{t=1}^{m_j} \Phi_{\frac{k_j}{(n^t, k_j)}}(x)^{d_{t,j}},
\end{equation}
where $h_i, k_j$ are positive integers, $a,b_i, c_j, d_j $ are non-negative integers, $m_j$ is the biggest integer $\le n$ such that $\frac{k_j}{(n^t, k_j)}$ is not coprime with $n$ and, if $c_j\neq 0$, then also $d_{t,j} \neq 0$ for every $t=1, \ldots, m_j$. Moreover
\begin{equation} \label{eq:degree}
 n= a+\sum_{i=1}^r b_i \varphi(h_i) + \sum_{j=1}^s \left ( c_j\varphi(k_j) + \sum_{t=1}^{m_j} d_{t,j} \varphi \left (\frac{k_j}{(n^{t},k_j)}\right ) \right ).
 \end{equation}
\end{theorem}

\begin{remark}
We notice that if $f$ is irreducible, $f$ has only one irreducible factors, so it is equal to some $\Phi_d(x)$ with $d\ge 3$, $n=\varphi(d)$ and $(n,d)=1$, as proved in Proposition \ref{prop:irreducible}.
\end{remark}

\begin{proof}
Assume first that $f$ is a monic polynomial of degree $n\ge 2$ satisfying the property \eqref{main_property}; as its roots are either $\alpha=0$ or $\alpha$ a root of unity then, then the irreducible factors of $f$ are either $x$ or cyclotomic polynomials. Assume that $\alpha \in \Qbar$ is a root of $f$ which is a primitive root of unity and denote by $k$ its order. As $\Phi_k(x)$ is the minimal polynomial of $\alpha$, then $\Phi_k \mid f$. If $(n,k)=1$, then $\alpha^n$ is again a primitive $k$-th root of unity, hence its minimal polynomial is still $\Phi_k$. Assume that $(n,k)\neq 1$ and let $m$ be the biggest integer $\le n$ such that $k/(n^m,k)$ is not coprime with $k$. Then, for every $t=1, \ldots, m$, we have that $\alpha^{n^t}$ is a primitive $k/(n^t,k)$-th root of unity, hence $\Phi_{k/(n^t,k)}$ must divide $f$. On the other hand, if $k/(n^t,k)$ is coprime with $n$, then $(n^t,k)=(n^m,k)$ and so $\alpha^{n^t}$ is a $k/(n^{m},k)$-th primitive root of unity for every $u=m+1, \ldots, n$, hence its minimal polynomial is again $\Phi_{k/(n^{m},k)}$. This implies that $f$ has to be of the form \eqref{eq:shape}. Moreover, the relation \eqref{eq:degree} comes directly by computing the degree of the product, recalling that $\deg \Phi_k(x)=\varphi(k)$. \\

We want to prove the converse; assume that $f$ is a polynomial of the shape \eqref{eq:shape}; then, the degree $n$ of $f$ satisfies \eqref{eq:degree}. Consider now a root $\alpha$ of $f$; then, $\alpha$ is a root of one of the irreducible factors of $f$. From \eqref{eq:shape}, then either $\alpha$ is equal to $0$, or $\alpha$ is a root of a cyclotomic polynomial, i.e. it is a root of unity.

If $\alpha=0$, then $\alpha^n=0$ and so $f(\alpha^n)=0$ as wanted. Assume now that $\alpha$ is a root of unity and denote by $k$ its order. As remarked before, if $\alpha$ has order $k$ then $\alpha^{n^t}$ has order $k/(n^t,k)$ for every $t=1, \ldots, n$. This implies that, if $(n,k)=1$ then $\alpha^{n^t}$ is a primitive $k$-root of unity, so $\Phi_k(\alpha^{n^t})=0=f(\alpha^{n^t})$. Assume now that $(n,k)\neq 1$; then, for every $t$ also $\Phi_{\frac{k}{(n^t,k)}} \mid f$ and so $f(\alpha^{n^t})=\Phi_{\frac{k}{(n^t,k)}}(\alpha^{n^t})=0$, as wanted.
\end{proof}

\begin{example}
For every $n \ge 1$, the polynomial $x^n-1$ satisfies the property \eqref{main_property}. Notice that
$$ x^n-1= \prod_{d \mid n} \Phi_d(x), $$
so $x^n-1$ has exactly the shape \eqref{eq:shape}.
\end{example}

\begin{example}
Let us list all the polynomials of degree $6$ which satisfy the property \eqref{main_property}.
By Theorem \ref{thm:general}, the possible factors of $f$ are either $x$ or cyclotomic polynomials of degree $\le 6$. We use a result of Gupta \cite{Gupta}, which asserts that if $n \in \varphi^{-1}(m)$, then $n < m < A(m)$ with
$$ A(m)= m \prod_{p-1 \mid m} \frac{p}{p-1}. $$
Using this, we have that $A(6)=21$, so we have to check the cyclotomic factors up to $\Phi_{21}$. An easy calculation gives that the cyclotomic polynomials with degrees equal to $1$ are $\Phi_1$ and $\Phi_2$, the ones of degree $2$ are $\Phi_i$ with $i=2,3,6$, the ones of degree $4$ are $\Phi_i$ with $i=5,8,9,12$ and the ones of degree $6$ are $\Phi_i$ with $i=7,14,18$. In order to classify the possible polynomials satisfying \eqref{main_property}, we have to take into account the fact that, if $\Phi_k \mid f$ for some $k$, then also $\Phi_{\frac{k}{(6,k)}} \mid f$.
The polynomials $f$ of degree $6$ that satisfy the property \eqref{main_property} have one of the following shape:
\begin{itemize}
   \item $f(x)=x^a\Phi_1(x)^b$ with $a+b=6$;
   \item $f(x)=x^a\Phi_1(x)^b\Phi_2(x)^c$ with $a+b+c=6$ and $b,c \neq 0$;
   \item $f(x)=x^a\Phi_1(x)^b\Phi_2(x)^c\Phi_3(x)^d$ with $a+b+c+2d=6$ and $b,d \neq 0$;
   \item $f(x)=x^a\Phi_1(x)^b\Phi_2(x)^c\Phi_3(x)^d\Phi_4(x)^{e}$ with $a+b+c+2d+2e=6$ and $b,d,e \neq 0$;
   \item $f(x)=x^a\Phi_1(x)^b\Phi_2(x)^c \Phi_5(x)$ with $a+b+c=2$;
   \item $f(x)=x^a \Phi_1(x)^b\Phi_2(x)^c\Phi_3(x)^d\Phi_4(x)^{e} \Phi_6(x)^f$ with $a+b+c+2(d+e+f)=6$ and $b,f \neq 0$;
   \item $f(x)=\Phi_7(x)$ (which is the only irreducible $f$);
   \item $f(x)=\Phi_1(x) \Phi_2(x) \Phi_{12}(x)$.
\end{itemize}
\end{example}

\section{Numbers coprime with their Euler totient function} \label{sec:Michon_numbers}

The integers $d$ such that $(d, \varphi(d))=1$ have been studied in number theory and appear in many different contexts; for example, these are the numbers such that there is only one group of order $d$ (i.e. the cyclic one). For this reason the numbers which satisfy this property are usually called \textit{cyclic}.


If $n$ is a prime, then $\varphi(n)=n-1$ so $n$ is a cyclic number; this shows that cyclic numbers are infinite. In \cite{Erdos}, Erd\"os gave an asymptotic formula for the number of cyclic numbers.

We prove the following easy necessary condition:
\begin{proposition} \label{prop:cyclic_char}
If $d$ is a cyclic number then either $d=2$ or $d$ is odd and squarefree.
\end{proposition}

\begin{proof}
We first notice that, if $d> 2$, then $\varphi(d)$ is even, so the only even number $d$ which is cyclic is $d=2$. On the other hand, assume that $p^2 \mid d$; then $p \mid \varphi(d)$, and so $(d, \varphi(d))\neq 1$ which contradicts the property of being cyclic.
\end{proof}

Of course these conditions are not sufficient in general; in fact, if we take for example $d=21$, then $\varphi(21)=12$ which is not coprime with $21$.

We can however prove the following result:

\begin{proposition}
 If $d$ is an odd numbers which is the product of two consecutive prime numbers, then $d$ is cyclic.
\end{proposition}

\begin{proof}
Assume that $d=p_n p_{n+1}$, where $p_i$ denotes the $i^{th}$ prime number; then we can prove that $p_n \nmid (p_{n+1}-1)$. In fact, assume by contradiction that $p_n \mid (p_{n+1}-1)$; as $p_n \neq 2$, then $p_n \mid \frac{p_{n+1}-1}{2}$, but this is a contradiction since by Bertrand's postulate $p_n < p_{n+1} < 2p_n$.
\end{proof}

Cyclic numbers are also related to Carmichael numbers. We recall that Carmichael numbers \cite{Carmichael} are composite numbers $n$ which satisfies the modular arithmetic congruence relation:
\[
b^{n-1} \equiv 1 \pmod n
\]
for all integers $b$ which are relatively prime to $n$.  Carmichael numbers are also called Fermat pseudoprimes or absolute Fermat pseudoprimes. Indeed, Carmichael numbers pass a Fermat primality test with respect to every base b relatively prime to the number, even though it is not actually prime. This makes tests based on Fermat's Little Theorem less effective than strong probable prime tests such as the Baillie-PSW primality test and the Miller-Rabin primality test. Korselt \cite{Korselt} proved that a positive composite integer $n$ is a Carmichael number if and only if $n$ is square-free, and for all prime divisors $p$ of $n$ then $p-1 \mid n-1$. For Carmichael numbers the following proposition holds:

\begin{proposition}
Every divisor of a Carmichael number is odd and cyclic.
\end{proposition}

In the '80, Michon conjectured that the converse is also true, i.e. that every odd cyclic number has at least one Carmichael multiple. The conjecture has been verified by Crump and Michon  for all the numbers $< 10000$, but remains still open.

\begin{example}
Here we list the cyclic numbers $1< d < 100$ which are not prime and the relative Euler totient functions:
\begin{tabbing}
marrone \= bla \= bla \= bla \= bla \= bla \= bla \= bla \= bla \= bla \= bla \= bla \= bla \kill
$d$ \>  15   \> 33   \> 35  \> 51   \> 65   \> 69  \> 77  \> 85 \> 87   \> 91  \> 95 \\
$\varphi(d)$ \> 8 \> 20 \> 24 \> 32 \> 48 \> 44 \> 60 \> 64 \> 56  \> 72 \> 72 \\
\end{tabbing}
This shows for example that for the numbers $\le 100$ the only cyclic numbers which have the same Euler totient function are $61$ and $77$ (for which $\varphi(61)= \varphi(77)=60$, see Example \ref{ese_sec4}) and $73, 91$ and $95$ (for which $\varphi(73)= \varphi(91)= \varphi(95)=72$).
\end{example}

\section{The irreducible case over $\F_p$} \label{sec:finite_fields}

The case of finite fields is very different from the previous one. Also in this case, all linear polynomials satisfy property \ref{main_property}, so without loss of generality we will always assume that $\deg f \ge 2$. In this setting, it is always true that, if $\alpha\in \F_q$ for some $q=p^a$, then $\alpha$ has finite order. We will denote by $k$ its order, i.e. the minimal $k$ such that ${\alpha}^k =1$. Notice that $(k,p)=1$. If we consider the cyclotomic polynomial $\Phi_k$,  it is not true anymore that this is the minimal polynomial of $\alpha$ as cyclotomic polynomials are not always irreducible over $\F_p$. Indeed, $\Phi_k$ factorises into $\varphi(k)/\ord_k(p)$ irreducible polynomials of degree $n:=\ord_k(p)$, where $\ord_k(p)$ denotes the multiplicative order of $k$ modulo $p$. This means that, if we consider the extension $\F_p(\alpha)/\F_p$, then it has degree $n$. Moreover, as this is a finite field extension, it is Galois with cyclic Galois group of order $n$, and a generator of the Galois group is the Frobenius, i.e. the automorphism with sends $\alpha \mapsto \alpha^p$. This means that all the conjugates of $\alpha$ are exactly $\alpha, \alpha^p, \ldots \alpha^{p^{n-1}}$. We are ready to prove the following:

\begin{proposition}
Let $f$ be an irreducible polynomial of degree $n\ge 2$; then, $f$ satisfies property \eqref{main_property} if and only if $f$ is a factor of degree $n$ of a cyclotomic polynomial $\Phi_k$ with $n < \varphi(k)$ and $n$ is a power of $p$, or $f=\Phi_k$, ${\ord}_k(p)=\varphi(k)$ and $(k, \varphi(k))=1$.
\end{proposition}

\begin{proof}
First assume that $f$ is an irreducible polynomial satisfying property \eqref{main_property}; then, if $\alpha$ is a root of $f$, either $\alpha$ is zero or $\alpha$ is a root of unity. If $\alpha=0$, then $f(x)=x$, which we exclude as we are assuming $\deg f \ge 2$.  Hence $\alpha$ is a primitive $k$-th root of unity for some $k\in \N$. Notice that $(k,p)=1$ since we are in characteristic $p$ and so $a^p=a$ for every $a\in \overline{\F}_p$. As $f$ is irreducible, we have that $f$ divides the cyclotomic polynomial $\Phi_k$. Now, if $\Phi_k$ is irreducible over $\F_p$ (which, as seen in Section \ref{sec:cyclotomic}, happens if and only if $\ord_k(p)=\varphi(k)$), then $f=\Phi_k$. Moreover, by property \eqref{main_property} we have that $\alpha^{{\varphi(k)}^t}$ must be a primitive $k$th root of unity for every $t \ge 1$, which implies that $(k,\varphi(k))=1$ as wanted.

Assume now that $\ord_k(p)<\varphi(k)$; then, $\Phi_k$ factorises into $\varphi(k)/\ord_k(p)$ irreducible factors of degree $\ord_k(p)$ and $f$ will be equal to one of these factors. This implies that $[\F_p(\alpha):\F_p]=\ord_k(p)$ and, using the Frobenius, all the other roots of $f$ will be $\alpha^p, \ldots, \alpha^{p^{k-1}}$. Therefore, as by property \eqref{main_property} we have that $\alpha^{\ord_k(p)}$ is a root of $f$, this implies that $\ord_k(p)$ must be a power of $p$, as wanted.
\medskip

Let us prove the converse. Assume first that $f=\Phi_k$ with $\ord_k(p)=\varphi(k)$ and $(k, \varphi(k))=1$; then, it is easy to prove that $f$ satisfies property \eqref{main_property} since $\alpha^{\varphi(k)^t}$ is again a primitive $k$-th root of unity as $(k, \varphi(k))=1$. Let us finally consider the case in which $f$ is an irreducible factor of degree $n< \varphi(k)$ and $n$ is a power of $p$; as seen in Section \ref{sec:cyclotomic}, we have that $n=\ord_k(p)$. As by assumption $n$ is a power of $p$, this directly implies that $\alpha^{n^t}$ is a root of $f$ since and all the cojugates of $\alpha$ are exactly $\alpha^p, \ldots \alpha^{p^{n-1}}$, concluding the proof.
\end{proof}

The last proposition gives strong constraints on the type of polynomials which satisfies property \eqref{main_property}; in particular:
\begin{itemize}
\item if $\Phi_k$ is irreducible over $\F_p$, then $(k,p)=1$ and $(k,\varphi(k))=1$;
\item if $\Phi_k$ is not irreducible over $\F_p$ and $(k,p)=1$, then $\ord_k(p)$ must be a power of $p$.
\end{itemize} 
In Section \ref{sec:cyclotomic} we recalled that $\Phi_k$ is irreducible over $\F_p$ if and only if $k=2,4, r^m$ or $2r^m$ with $r$ an odd prime, and the multiplicative order of $p$ modulo $k$ is maximal. By Proposition  \ref{prop:cyclic_char}, we have that if $(k,\varphi(k))=1$ then either $k=2$ or $k$ is odd and squarefree; combining these two conditions we have that $\Phi_k$ is irreducible over $\F_p$ with $(k,\varphi(k))=1$ if and only if either $k=2$ and $p \neq 2$ or $k$ is an odd prime different from $p$ and $\ord_k(p)=k-1$.\\ 

In the second case we have even a more restricted condition; indeed, $\ord_k(p)$ is a divisor of $\varphi(k)$. Let us write $k=\prod_{i=1}^m p_i^{a_i}$ where the $p_i$ are distinct primes and $a_i$ are positive integers; then, by definition  $\varphi(k)=\prod_{i=1}^m p_i^{a_i-1}(p_i-1)$. As we are assuming that $\ord_k(p)$ is equal to some power of $p$, this implies that $p \mid \varphi(k)$, hence $p \mid (p_i-1)$ for some $i$. But it is clear that this can happens if and only if $p=2$.

Using these considerations, we proved the following result:

\begin{theorem}
Let $f$ be an irreducible polynomial over $\F_p$ of degree $n\ge 2$; then, $f$ satisfies property \eqref{main_property} if and only if either $f(x)=x^{r-1}+ \cdots +1$ with $r$ a prime different from $p$ and $p$ a generator of $(\Z/r\Z)^{\times}$ or $p=2$ and $f$ is a factor of degree $n$ of a cyclotomic polynomial $\Phi_k$ with $n < \varphi(k)$, where $k$ is odd and and $n=\ord_k(2)$ is a power of $2$.
\end{theorem}

\begin{example}
Let us consider the cyclotomic polynomial of order $15$ over $\F_2$; as seen before, we have that  $\Phi_{15}$ factorises into $2$ factors of degree $4$, i.e.
\[
\Phi_{15}(x)=(x^4+x^3+1)(x^4+x+1).
\]
In this case both the factors of $\Phi_{15}$ are polynomials of degree $4$ satisfying property \eqref{main_property}. We point out that if $\Phi_n$ factorises into irreducible factors over $\F_2$ then either none of the factors satisfy the property or all do.
\end{example}

\bibliographystyle{amsalpha}

\end{document}